\documentclass[12pt]{amsart}
\usepackage{graphicx} % Required for inserting images
\usepackage[a4paper,margin=3cm]{geometry}%this is for easy margins
\usepackage{amsmath,verbatim}
\usepackage{enumerate}
\usepackage{appendix}
\usepackage[T1]{fontenc}
\usepackage[utf8]{inputenc}
\usepackage{adjustbox}
\usepackage{graphicx}
\usepackage{dsfont}
\usepackage{lineno}
\usepackage{hyperref} 
\usepackage[section]{placeins}
\hypersetup{hidelinks,
colorlinks=true,
linkcolor=black,
citecolor=black}
\usepackage{mathtools}
\usepackage{stix}
\usepackage{array}
\usepackage{amsmath,amsthm,amsfonts}
\usepackage[shortlabels]{enumitem}
\usepackage{rotating}
\usepackage{color}%cyan,red;magenta,green;yellow,blue
\usepackage{caption}
\newcolumntype{T}[1]{S[table-format=#1]}
\newcolumntype{U}[1]{S[table-format=#1,
                       round-mode=places, % apply rounding automatically
                       round-precision=2]}
\usepackage{marginnote}%defines nice marginal notes

\newcommand{\real}{\mathbb{R}}

\newcommand{\I}{\mathbb 1}

\newcolumntype{?}{!{\vrule width 0.7pt}}

\usepackage{longtable}
\newcommand{\plus}{\scalebox{0.75}[1.0]{$+$}}
\newcommand{\minus}{\scalebox{0.75}[1.0]{$-$}}
\renewcommand{\Re}{\ensuremath{\operatorname{Re}}}
\renewcommand{\Im}{\ensuremath{\operatorname{Im}}}

\usepackage{color}%for comments

\newcommand{\normal}{\color{black}}

\numberwithin{equation}{section}
\theoremstyle{plain}
\newtheorem{theorem}{Theorem}
\newtheorem{lemma}[theorem]{Lemma}

\newtheorem{corollary}[theorem]{Corollary}

\theoremstyle{definition}
\newtheorem{remark}[theorem]{Remark}
\newtheorem{definition}[theorem]{Definition}
\newtheorem{example}[theorem]{Example}

\numberwithin{theorem}{section}
\usepackage{microtype}
\usepackage{soul}

\sodef\myspace{}{.2em}{1em plus1em}{2em plus.1em minus.1em}

\newcommand{\rli}[3]{\sideset{_{\,#1}^{\mathrm{R}}}{_{#2}^{#3}}{\mathop{\mathrm{I}}}}  % Riemann--Liouville Integral

\newcommand{\rld}[3]{\sideset{_{\,#1}^{\mathrm{R}}}{_{#2}^{#3}}{\mathop{\mathrm{D}}}}
\newcommand{\mad}[3]{\sideset{_{#1}^{\mathrm{M}}}{_{#2}^{#3}}{\mathop{\mathrm{D}}}}
\newcommand{\cd}[3]{\sideset{_{\,#1}^{\mathrm{C}}}{_{#2}^{#3}}{\mathop{\mathrm{D}}}} 
\newcommand{\ced}[3]{\sideset{_{~\,#1}^{\mathrm{Ce}}}{_{#2}^{#3}}{\mathop{\mathrm{D}}}} %Censored Derivative
\newcommand{\ed}[3]{\sideset{_{\,#1}^{\mathrm{E}}}{_{#2}^{#3}}{\mathop{\mathrm{D}}}} 
\allowdisplaybreaks

 \makeatletter
\g@addto@macro{\endabstract}{\@setabstract}
\newcommand{\authorfootnotes}{\renewcommand\thefootnote{\@fnsymbol\c@footnote}}%
\makeatother
\begin{document}
\title[The mapping properties in weighted fractional Sobolev space]{\bf{ The mapping properties of fractional derivatives in weighted fractional Sobolev space}}
%\author{Cailing Li}
\maketitle
\begin{center}
\authorfootnotes
  Cailing Li%\footnote{Author A}
  \textsuperscript{} 
%  David Berger%\footnote{Author B}
 % \textsuperscript{1},
  %Ren\'e L.Schilling 
  \footnote{cailingli.math@gmail.com \\  \quad \quad This paper is part of Cailing Li's PhD thesis \cite{2023_Li} which was written at Technische Universität Dresden 
 under the supervision of Ren\'e L.Schilling. The author is grateful to David Berger for his invaluable discussion.} %Author %D\footnote{Author %D}\textsuperscript{2} and
  %Author E\footnote{Author E}\textsuperscript{2} \par \bigskip

  \textsuperscript{}Institute of Mathematical Stochastics, Technische Universität Dresden, D-01217 Dresden, Germany \par \bigskip
%\author{Cailing Li, David Berger, Ren\'e L. Schilling }
%\affil{TeX.SX}
\end{center}
%\begin{center}
%\authorfootnotes
%  Cailing Li%\footnote{Author A}
%  \textsuperscript{} 
%%  David Berger%\footnote{Author B}
% % \textsuperscript{1},
%  %Ren\'e L.Schilling 
%  \footnote{cailingli.math@gmail.com \\  \quad \quad This paper is part of Cailing Li's PhD thesis \cite{2023_Li} which was written at Technische Universität Dresden 
% under the supervision of Ren\'e L.Schilling. The author is grateful to David Berger for his invaluable discussion.}
 
 %\textsuperscript{}Institute of Mathematical Stochastics, Technische Universität Dresden, D-01217 Dresden, Germany \par \bigskip
%\author{Cailing Li, David Berger, Ren\'e L. Schilling }
%\affil{TeX.SX}
%\end{center}

\section*{Abstract}
%This paper is aimed at the readers interested in the information on mapping properties of Riemann--Liouville fractional derivative on Sobolev space.
%In general, when we take the fractional derivatives of a function, we may lose regularity, while fractional integrals to the contrary enhancer them. 
% With the development of functional analysis, these basic facts have led to numerous discussions about the behavior of fractional integrals and derivatives in various function spaces. 
We study the mapping behavior of the Marchaud fractional derivative with different extensions in the scale of fractional weighted Sobolev spaces. In particular we show that the $\alpha$--order Riemann--Liouville fractional derivative maps $W^{p,s}_0(\Omega)$ to $W^{p,s-\alpha}(\Omega)$, for all $0<\alpha<s<1$ and the $\alpha$--order Marchaud fractional derivative with even extension maps the fractional Sobolev space  $W^{p,s}((0,\infty))$ to $W^{p,s-\alpha}(\real)$ for all $0<\alpha<s<1$ and $ps\geq1$ . The proof is based on the Calder\'{o}n--Lions interpolation theorem.
%This paper specifically focuses on the behavior of Riemann--Liouville fractional derivatives in Sobolev spaces. We show that $\alpha$ order Riemann--Liouville fractional derivative maps $W^{p,s}_0$ to $W^{p,s-\alpha}$, for all $0<\alpha<s<1$, using Calder\'{o}n Lions Interpolation Theorem.
\section{Introduction}
Fractional derivatives have recently become important tools to model real world phenomena. There are important applications in physics, chemistry and biology %(so--called “master equations” for “anomalous diffusions”) 
as well as in mathematics. A good introduction to applications is given in the monograph Klages el al \cite{2008_Klages}.

One of the most frequently used extensions of the usual derivative is the
\textbf{Riemann--Liouville fractional derivative} of order $\alpha\in (0,1)$ on $[0,x]$ 
    \begin{align*}
      \rld{0}{x}{\alpha} \phi = \frac 1{\Gamma(1-\alpha)} \frac d{dx} \int_0^x \frac{\phi(s) }{(x-s)^{\alpha}}\,ds, \quad \alpha\in (0,1).
    \end{align*} 
It is defined in such a way that it “fills the gaps” between the zero-order “derivative” (i.e.  $\phi = \left(\frac d{dx}\right)^0\phi$) and the usual derivative (i.e.  $\phi'=\left(\frac d{dx}\right)^1\phi$).  Note that  $\rld{0}{x}{\alpha}$ is a non local integral operator  if $\alpha \notin \mathbb{N}_0$. 
  
 \textbf{Fractional derivative in the sense of Caputo} is given by
$$
     \cd{0}{x}{\alpha} \phi = \frac 1{\Gamma(1-\alpha)} \frac d{dx}\int_0^x \frac{\phi(s)-\phi(0)}{(x-s)^\alpha}\,ds,
    \quad x\geq 0,\;\alpha\in(0,1).
$$

The Caputo--type fractional derivative is particularly useful when dealing with initial value problems, as it takes into account the initial condition(s) of the function being differentiated. It is widely used in fractional calculus and has applications in various fields, including engineering, and finance, we refer the reader to Meerschaert et al \cite{2009_Meerschaert},  Baeumer  and  Meerschaert \cite{2001_Baeumer} and Eidelman et al \cite{2004_Eidelman}. 

It is important to note that there are differences between the Riemann--Liouville fractional derivative and the Caputo--type fractional derivative. The main difference lies in their treatment of the initial conditions. The Caputo--type derivative considers the initial conditions of the function, while the Riemann--Liouville derivative does not.   The relationship between these two derivatives is a subject of study in fractional calculus. There is a substantial literature demonstrating the connection and equivalence between Riemann--Liouville and Caputo--type fractional derivatives. 

These papers provide conditions such that these two derivatives are equal or related to each other. Details and proofs can be found in the relevant literature on fractional calculus, we refer the reader to Samko et al \cite{1993_Samko}, Kochubei et al \cite{2019_Kochubei}.

In the case of a smooth function that vanishes outside of $(0,\infty)$, it has been shown in Du et.al \cite{2021_Du} that the Riemann--Liouville fractional derivative is equal to the Caputo derivative. This equivalence is particularly significant as it corresponds to the generator of a killed subordinator.  This leads us to consider the fractional derivative as  the generator of a Markov process. The \textbf{Marchaud fractional derivative} is given by
\begin{align}
\label{mar-def-001}
    \mad{}{+}{\alpha} \phi(x)&=\frac{\alpha}{\Gamma(1-\alpha)}\int_{0+}^\infty \frac{\phi(x)-\phi(x-s)}{s^{1+\alpha}}\,ds, \quad x\in \real, \quad \phi: \real\to \real.
    %\\&=\frac{\alpha}{\Gamma(1-\alpha)}\int_{-\infty}^t \frac{u(t)-u(s)}{(t-s)^{1+\alpha}}\,ds
\end{align}
 We will use the extension expressed Riemann--Liouville and Caputo fractional derivatives in Section 2.  

The Marchaud derivatives are \emph{a priorily} defined on the whole real axis  that extends the concept of differentiation to non-integer orders. It was introduced by Pierre Marchaud in the early 20th century and has found applications in various areas of mathematics and physics, we refer  the reader  to Samko et al \cite{1993_Samko}. The Marchaud derivative differs from the Riemann--Liouville and Caputo derivatives   since it does not explicitly use derivatives. It has some distinct properties and behavior, which make it suitable for certain applications, especially in probability,   see  Hernández-Hernández and Kolokoltsov  \cite{ 2015_HernandezHernandez} and Kolokoltsov \cite{2019_Kolokoltsov}.

% The concept of fractional derivative is closely connected to Abel's integral equation.
%\begin{equation*}
%\frac1{\Gamma(\alpha)}\int_0^x \frac{\phi(t)}{(x-t)^{1-\alpha}}dt=g(x).
%\end{equation*}
%Combined the Abel's integral and following Cauchy formula we get the definition of Riemann--Liouville fractional derivative. 
%\begin{align*}
%\int_0^x \frac{g(t)}{(x-t)^\alpha}\,dt &= \frac1{\Gamma(\alpha)}\int_0^x\int_0^t \frac{\phi(s)}{(t-s)^{1-\alpha}}\,ds\,\frac{dt}{(x-t)^\alpha}
%\\&=\frac1{\Gamma(\alpha)}\int_0^x\int_s^x\frac{\phi(s)}{(t-s)^{1-\alpha}}\,\frac{dt}{(x-t)^\alpha}\,ds
%\\&=\Gamma(1-\alpha)\int_0^x \phi(s)\,ds.
%\end{align*}
%The connection between fractional derivatives and Abel's integral equation arises from the fact that the Riemann--Liouville fractional derivative can be used to solve certain types of integral equations, including Abel's integral equation. The details of this connection and further properties of fractional derivatives can be found in relevant literature on the topic, we refer reader Samko et al \cite{1993_Samko} and Samko, Cardoso \cite{2003_Samko}.
It is a well--known result that the differential operator and integral operator have a reciprocal relation when the latter is applied first. In other words, $\frac{d}{dx}\int_0^x \phi(t)\,dt = \phi(x)$. The Riemann--Liouville fractional derivative allows us to generalize the notion of differentiation to non-integer orders. It has various applications in applied mathematics, particularly in the study of systems with memory and non-local behavior, see Kilbas et al \cite{2006_Kilbas}, Diethelm \cite{2010_Diethelm}. Similarly, one might hope that the composition of the Riemann--Liouville fractional derivative and integral operators satisfy the same relation.

As the right inverse of the Riemann--Liouville fractional derivative, the \textbf{Riemann--Liouville  fractional integral} of order $\alpha\in (0,1)$ defined  on $[0,x]$ is given by
    \begin{align*}
        \rli{0}{x}{\alpha}\phi= \int_0^x \phi(t) \frac{(x-t)^{\alpha-1}}{\Gamma(\alpha)}\,dt, \quad x\in [0,\infty). %,\\
        %\rli{x}{\infty}{\alpha}\phi &= \int_x^\infty \phi(t) \frac{(t-x)^{\alpha-1}}{\Gamma(\alpha)}\,dt.
    \end{align*}
 With the development of functional analysis, these basic facts have led to numerous discussions about the behavior of fractional integrals and derivatives in various function spaces.  
The mapping properties of various forms of fractional integration operators in known function spaces is of great interest see \cite{2016_Rafeiro, 2017_Bergounioux, 2020_Carbotti}. Fractional integration shares some smoothness with classical integrals, increases smoothness of functions in certain ways. On the other hand, fractional differentiation often decrease smoothness the properties of functions. When we discuss mapping properties, we refer to the following question: given a specific function space $\mathcal{X}$,  can we characterize  $\rli{0}{\cdot}{\alpha} \mathcal{X}=\{\rli{0}{\cdot}{\alpha}\phi, \phi\in \mathcal{X}\}$ as a subspace of another function space $\mathcal{Y}$.

Hardy and Littlewood made significant contributions to the study of mapping properties for the Riemann-Liouville fractional integral $\rli{0}{\cdot}{\alpha}$. Their results primarily focus on the integrability properties of functions and the continuity properties of functions.
Such statements  are known as \textbf{Hardy--Littlewood theorems}. We refer the reader to \cite[P.66, P.103 ]{1993_Samko}. 

Rafeiro and Samko \cite{2016_Rafeiro} investigated the behavior of Riemann--Liouville fractional integrals and derivatives in various function spaces, such as H\"older, Lebesgue, and Morrey spaces. However, their work did not explore these behaviors specifically in Sobolev space. 

%Over the past two decades, significant progress has been made in understanding the properties of spaces characterized by variable exponents. For a more detailed discussion on variable exponents, we refer to \cite{2013_Uribe, 2011_Diening}. %It might become a hot topic to study the mapping properties of fractional integrals and derivatives in the variable exponents, such as 
%The investigation of mapping properties in fractional integrals and derivatives for spaces with variable exponents, $L^{p(\cdot)}, C^{\beta(\cdot)}$ the latter being a type of H\"older space, is might to become a trending area of research in the near future.

Carbotti et al. \cite{2020_Carbotti} showed that the $\alpha$-order fractional derivative maps the Sobolev space $W^{p,1}_0$  (see P.\pageref{sob-0} for the definition) to the fractional Sobolev-Slobodeckij space $W^{p,\alpha}$ for $\alpha\in (0,\min\{1/p,(p-1)/2p\})$, which implies that $\alpha<1/2$. Understanding from differential operator and integral operator, it is commonly expected that applying an $\alpha$-order fractional integral to a function may increase its regularity by $\alpha$, whereas an $\alpha$-order fractional derivative might decrease the regularity  by $\alpha$.  However, in the work of Carbotti et al., applying an $\alpha$-order derivative results in a loss of 1-$\alpha$ regularity. Note that $\alpha<1/2$ and $1-\alpha>\alpha$ that means loss of regularity too large. %This paper addresses this gap, presenting a new finding (see Theorem \ref{map-frc-2.2.9}). 
%Usually,  when we apply $\alpha$ order fractional integral to a function, we might hope this function gains $\alpha$ order regularity, in the contrast, when apply $\alpha$ order fractional derivative, we might hope this function lose $\alpha$ order regularity.  While in Carbotti et al. \cite{2020_Carbotti}, when apply $\alpha$ order derivative to a function, this function lose $1-\alpha$ regularity. 

%This paper is aimed at the readers interested in the information on mapping properties of Riemann--Liouville fractional derivative on Sobolev space.
%In general, when we take the fractional derivatives of a function, we may lose regularity, while fractional integrals to the contrary enhancer them. 
 In this paper,  we extend the results appeared in \cite{2020_Carbotti} and we show in this paper Corollary \ref{rl-3-8} that $\alpha$ order Riemann--Liouville fractional derivative maps $W^{p,s}_0$ to $W^{p,s-\alpha}$, for all $0<\alpha<s<1$ and $\alpha$ order Marchaud fractional derivative with even extension maps weight fractional Sobolev space  $W^{p,s}((0,\infty))$ to $W^{p,s-\alpha}(\real)$ for all $0<\alpha<s<1$ and $ps\geq1$ in Corollary \ref{rl-3-9}. using Calder\'{o}n Lions Interpolation Theorem.  %Which seems to be more natural and includes the above result, even extending it. 

\section{Basic definition and notation}
\subsection{Fractional derivatives}
In this section, we will provide some tools to investigate classical integrals and derivatives, starting with the general form of the Cauchy formula: 
%The Riemann-Liouville fractional integral is derived from the Cauchy formula. Assuming $\phi$ is integrable in $(a, t)$ and denoting $\phi^{(-1)}$ as the integral of $\phi$, we have:
%
%\begin{align*}
%\phi^{(-1)}(t) = \int_a^t \phi(s) \, ds.
%\end{align*}
%
%We can extend this concept further by considering iterated integrals. For example, the two-fold integral is given by:
%
%\begin{align*}
%\phi^{(-2)}(t) = \int_a^t \int_a^{s_1} \phi(s_2) \, ds_2  \,ds_1= \int_a^t (t-s)\phi(s) \, ds.
%\end{align*}
%
%By induction, we obtain 

\begin{equation}\label{cau-for-1}
\phi^{(-n)}(t) = \frac{1}{\Gamma(n)} \int_a^t (t-s)^{n-1} \phi(s) \, ds.
\end{equation}

If we replace the integer $n$ in the Cauchy formula \eqref{cau-for-1} with a real number $\alpha$, we obtain an integral of arbitrary order. This leads us to the definition of the Riemann-Liouville fractional integral, denoted by $\rli{0}{\cdot}{\alpha}$, which is discussed in detail in works by Podlubny and Samko et al. \cite{1998_Podlubny, 1993_Samko}.

\begin{definition}[Riemann--Liouville fractional integral]\label{d:rli}
    The \textbf{Riemann--Liouville  fractional integral} of order $\alpha\in (0,1)$ defined  on $[0,x]$ %and $[x,\infty)$, respectively, 
    is given by
    \begin{align*}
        \rli{0}{x}{\alpha}\phi &= \int_0^x \phi(t) \frac{(x-t)^{\alpha-1}}{\Gamma(\alpha)}\,dt,\quad x\in [0,\infty). %,\\
       % \rli{x}{\infty}{\alpha}\phi &= \int_x^\infty \phi(t) \frac{(t-x)^{\alpha-1}}{\Gamma(\alpha)}\,dt.
    \end{align*}
    %The Riemann--Liouville  fractional integral of order $\alpha>0$ on $[x,0]$ and $(-\infty,x]$, respectively, is
    %\begin{align*}
       % \rli{x}{0}{\alpha}\phi &= \int_x^0 \phi(t) \frac{(t-x)^{\alpha-1}}{\Gamma(\alpha)}\,dt,\\
       % \rwli{-\infty}{x}{\alpha}\phi &= \int_{-\infty}^x \phi(t) \frac{(x-t)^{\alpha-1}}{\Gamma(\alpha)}\,dt.
    %\end{align*}
\end{definition}
 %We will provide the definition of the Riemann-Liouville fractional derivative.  

Next we will give several definition of fractional derivative.  In this paper we restrict ourselves to $\alpha \in (0,1)$ since this is the only case we are interested in. 

To define the Riemann--Liouville fractional derivative, let $\phi$ be a function defined on $(a, b)$.
 \begin{definition}\label{rl-fra-def}%[Riemann--Liouville fractional derivative]
    \textbf{The Riemann--Liouville fractional derivative} of order $\alpha\in (0,1)$ on $[0,x]$ %and $[x,\infty)$, respectively, 
    is defined by 
    \begin{align*}
        \rld{0}{x}{\alpha}\phi &= \frac 1{\Gamma(1-\alpha)} \frac d{dx} \int_0^x \frac{\phi(t) }{(x-t)^{\alpha}}\,dt=\frac d{dx}\rli{0}{x}{1-\alpha}\phi. %,\\
        %\rld{x}{\infty}{\alpha}\phi &= -\frac 1{\Gamma(1-\alpha)} \frac d{dx} \int_x^\infty \frac{\phi(t) }{(t-x)^{\alpha}}\,dt=\frac d{dx}\rli{x}{\infty}{1-\alpha}\phi.
    \end{align*}
    Next, we will introduce the Caputo--type fractional derivative. The Caputo fractional derivative is another approach to generalize the concept of differentiation to non--integer orders. We also provide the definition of censored fractional derivative, which is studied in \cite{2021_Du}. 

\begin{definition}\label{cap-2-1-3}
Let $\phi:[0,\infty)\to \real$ be differentiable and $\alpha\in (0,1)$. The
\textbf{Caputo fractional derivative} of order $\alpha\in (0,1)$ is defined as
    \begin{equation}\label{fra-cap}
        \cd{0}{x}{\alpha}\phi = \frac 1{\Gamma(1-\alpha)} \int_0^x \frac{\phi'(t)}{(x-t)^\alpha}\,dt=\rld{0}{x}{\alpha}\phi-\phi(0)\frac{x^{-\alpha}}{\Gamma(1-\alpha)}.
    \end{equation}
\end{definition}
\begin{definition}\label{cen-fra-216}   Let $\phi: [0,\infty)\to \real$ be a function, then the \textbf{censored fractional derivative} of order $\alpha$, $\alpha\in (0, 1)$, is defined as
 \begin{equation}
\begin{aligned}
 \ced{0}{x}{\alpha}\phi(x)%&=\frac{\alpha}{\Gamma(1-\alpha)}\int_0^x \left(\phi(x)-\phi(x-s)\right) s^{-\alpha-1}\,ds
=\rld{0}{x}{\alpha}\phi-\phi(x)\frac{x^{-\alpha}}{\Gamma(1-\alpha)}.
\end{aligned}
 \end{equation}
\end{definition}
% The Riemann--Liouville fractional derivative of order $\alpha\in (0,1)$ on $[x,0]$ and $(-\infty,x]$, respectively, is given by 
    %\begin{align*}
       % \rld{x}{0}{\alpha}\phi &=-\frac 1{\Gamma(1-\alpha)} \frac d{dx} \int_x^0 \frac{\phi(t) }{(t-x)^{\alpha}}\,dt=\frac d{dx}\rli{x}{0}{1-\alpha}\phi,\\
       % \rwld{-\infty}{x}{\alpha}\phi &=\frac 1{\Gamma(1-\alpha)} \frac d{dx} \int_{-\infty}^x \frac{\phi(t) }{(x-t)^{\alpha}}\,dt=\frac d{dx} \rwli{-\infty}{x}{1-\alpha}\phi.
    %\end{align*}
    \end{definition}
    We will now see how we can treat the various fractional derivatives in a unified way. Let us introduce the definition of Marchuaud fractional derivative.

\begin{definition}[Marchaud fractional derivative]\label{mar-def} Let $u:\real\to\real$ be a function which is defined on the whole real axis and $\alpha\in (0, 1)$. The \textbf{Marchaud fractional derivatives} are the following operators
\begin{align}
\label{mar-def+}
    \mad{}{+}{\alpha}u(x)
    &:= \frac{\alpha}{\Gamma(1-\alpha)}\int_{0+}^\infty \frac{u(x)-u(x-t)}{t^{1+\alpha}}\,dt,
  % \\
%
%\label{mar-def-}
    %\mad{}{-}{\alpha}u(x)
    %&:= \frac{\alpha}{\Gamma(1-\alpha)}\int_{0+}^\infty \frac{u(x)-u(x+t)}{t^{1+\alpha}}\,dt.
\end{align}
\end{definition}
Let us introduce the following important definition of extension and restriction of an operator, which taken from \cite[P.142]{2012_Kato}. 
\begin{definition}\label{extension operator}
Let $A, B: X\to Y$ be two operators such that $\mathcal{D}(A)\subset\mathcal{D}(B)$ and $A u=Bu$ for all $u\in \mathcal{D}(A)$, $B$ is called an extension of $A$ denoted as $ext_{\mathcal{D}(B)} A=B$ and $A$ a restriction of $B$ denoted as $res_{\mathcal{D}(A)} B=A$.

%An extension of $A$ to a superset $E$ of $X$, i.e $X\subseteq E$ is an operator $ext_E A: E\to Y$ such that for all $ext_E A=A$ on $X$.
%A linear extension operator is a linear map $T: X \to Y$ such that $E\subset X$, $T_E: E\to Y$ and $T_E=T$ on $E$.

%A restriction operator of $A$ to $Z$ as denote $res_Z A=A|_Z$ and $Z\subseteq X$, is an operator $A|_Z: Z\to Y$ defined as $A|_Z=A$ on $Z$.
\end{definition}

Let us now investigate Marchaud derivatives for functions which are only defined on the half-axis $[0,\infty)$. When comparing Marchaud derivatives with fractional derivatives on $[0,\infty)$, we have to extend functions defined on the half-axis onto $(-\infty,0]$.
\begin{definition}[Extension]\label{mar-ext}
    Let $\phi : [0,\infty)\to\real$. Then we can extend $\phi$ in the following way to become a function on $\real$
    \begin{itemize}
    \item `Killing type' extension: $\phi^0(x)=\phi(x)\I_{[0,\infty)}(x) + 0 \cdot \I_{(-\infty,0)}(x)$.
    \item `Sticky type' extension: $\phi^\sigma(x) := \phi(x)\I_{[0,\infty)}(x) + \phi(0)\I_{(-\infty,0)}(x)$.
    \item `Even' extension: $\phi^e (x) := \phi(x)\I_{[0,\infty)}(x) + \phi(-x)\I_{(-\infty,0)}(x)$.
     %\item `Shift' extension: $\phi^s (x, y) := \phi(x)\I_{[0,\infty)}(x) + \phi(x+y)\I_{(-\infty,0)}(x)$, for  all $y\in (0, \infty)$.
    \end{itemize}
    \end{definition}
    \begin{remark}
        If we take $X=[0, \infty), E=\real$, we have $ext_\real^0 \phi=\phi^0$ and $ext_\real^\sigma \phi=\phi^\sigma$ are two extension operator.
    \end{remark}
    \begin{definition} Let $\alpha\in(0,1)$ and $\phi :[0,\infty)\to \real$ be a function. The even extension of the Marchaud fractional derivative is given by  
    \begin{align*}\mad{}{+}{\alpha}\phi^e(x)
        &= \frac{\alpha}{\Gamma(1-\alpha)}\int_0^\infty \frac{\phi^e(x)-\phi^e(x-s)}{s^{\alpha+1}}\,ds
        = \ed{0}{x}{\alpha}\phi,\quad x\in \real.
        \end{align*}
        \end{definition}
    The connection between Riemann--Liouville, Caputo and Marchaud derivatives is  as follows:
\begin{lemma}\label{mar-rl-frac}
    Let $\alpha\in(0,1)$ and $\phi :[0,\infty)\to \real$ be a function.
    \begin{align*}
        \rld{0}{x}{\alpha}\phi
        &= \frac{\alpha}{\Gamma(1-\alpha)}\int_0^\infty \frac{\phi^0(x)-\phi^0(x-s)}{s^{\alpha+1}}\,ds
        = \mad{}{+}{\alpha}\phi^0(x),\\
        \cd{0}{x}{\alpha}\phi
        &= \frac{\alpha}{\Gamma(1-\alpha)}\int_0^\infty \frac{\phi^\sigma(x)-\phi^\sigma(x-s)}{s^{\alpha+1}}\,ds
        = \mad{}{+}{\alpha}\phi^\sigma(x)
%        \\ \ed{0}{x}{\alpha}\phi
%        &= \frac{\alpha}{\Gamma(1-\alpha)}\int_0^\infty \frac{\phi^e(x)-\phi^e(x-s)}{s^{\alpha+1}}\,ds
%        = \mad{}{+}{\alpha}\phi^e(x)%\\
        %\ced{0}{x}{\alpha}\phi
        %&= \frac{\alpha}{\Gamma(1-\alpha)}\int_0^\infty \frac{\phi^s (x)-\phi^s (x-s)}{s^{\alpha+1}}\,ds
        %= \mad{}{+}{\alpha}\phi^s(x),\\
       % \cd{0}{x}{\alpha}\phi+\cd{x}{\infty}{\alpha}\phi
       % & = \frac{\alpha}{\Gamma(1-\alpha)}\int_0^\infty \frac{\phi^e (x)-\phi^e(x-s)}{s^{\alpha+1}}\,ds
        %= \mad{}{+}{\alpha}\phi^e (x).
    \end{align*}
    %whereas the formula
   % \begin{equation*}
    %\cd{x}{\infty}{\alpha}\phi =
        %\rld{x}{\infty}{\alpha}\phi
        %= \frac{\alpha}{\Gamma(1-\alpha)}\int_0^\infty \frac{\phi(x)-\phi(x+s)}%{s^{\alpha+1}}\,ds
        %= \mad{}{-}{\alpha}\phi(x)
    %\end{equation*}
    is valid for any extension of $\phi$ \textup{(}as we do not need to extend $\phi$!\textup{)}.
\end{lemma}
\normal
\begin{proof}
    We begin with the integral expression in the middle of the first displayed formula. As usual, we assume that $\phi$ is as regular as we need it for the calculation below. Using the definition of the killing-type extension we see
    \begin{align*}
        \frac{\alpha}{\Gamma(1-\alpha)}&\int_0^\infty \frac{\phi^0(x)-\phi^0(x-s)}{s^{\alpha+1}}\,ds\\
        &= \frac{\alpha}{\Gamma(1-\alpha)}\left\{ \int_0^x \frac{\phi(x)-\phi(x-s)}{s^{\alpha+1}}\,ds + \int_x^\infty \frac{\phi(x)-0}{s^{\alpha+1}}\,ds \right\}\\
        &= \frac{1}{\Gamma(1-\alpha)}\left\{ \alpha \int_0^x\int_0^s \phi'(x-t)\,dt \frac{ds}{s^{\alpha+1}}+ (\phi(x)-0) x^{-\alpha} \right\}\\
        &= \frac{1}{\Gamma(1-\alpha)}\left\{ \int_0^x\int_t^x \frac{\alpha \, ds}{s^{\alpha+1}}\,\phi'(x-t)\,dt + (\phi(x)-0) x^{-\alpha} \right\}
%\\&= \frac{1}{\Gamma(1-\alpha)}\left\{ \int_0^x \left[-s^{-\alpha}\right]_t^x \,\phi'(x-t)\,dt + (\phi(x)-0) x^{-\alpha} \right\}
\\ &= \frac{1}{\Gamma(1-\alpha)}\left\{ \int_0^x t^{-\alpha}\,\phi'(x-t)\,dt - x^{-\alpha}\int_0^x \phi'(x-t)\,dt + (\phi(x)-0) x^{-\alpha} \right\}\\
        &= \frac{1}{\Gamma(1-\alpha)}\left\{ \int_0^x \frac{\phi'(x-t)}{t^\alpha}\,dt - \big(\phi(x)-\phi(0)\big)x^{-\alpha} + (\phi(x)-0) x^{-\alpha}\right\}\\
        &= \frac{1}{\Gamma(1-\alpha)}\left\{ \int_0^x \frac{\phi'(x-t)}{t^\alpha}\,dt + \frac{\phi(0)-0}{x^\alpha}\right\}\\
        &= \rld{0}{x}{\alpha}\phi
    \end{align*}
    where we use Definition \ref{cap-2-1-3} in the last equality.

    Note that the artificial "$-0$" should remind of the fact that we use the killing extension $\phi^0$. If we use the sticky extension $\phi^0 \rightsquigarrow \phi^\sigma$ and if we change throughout the above calculation $-0\rightsquigarrow -\phi(0)$, then we get the second claimed formula for the Caputo derivative.

    %Let us now consider the even extension $\phi^e$. Note that
    % \begin{align*}
         %\int_x^\infty \frac{\phi(x)-\phi(s-x)}{s^{\alpha+1}}\,ds
         %&=-\int_x^\infty\left(\int_x^s \phi'(t-x) dt+(\phi(0)-\phi(x))\right)s^{-\alpha-1}\,ds\\
         %&=-\int_x^\infty\int_x^s \phi'(t-x) dt\,s^{-\alpha-1}\,ds-\frac{\phi(0)-\phi(x)}{\alpha x^{\alpha}}\\
         %&=-\int_x^\infty\int_t^{\infty}s^{-\alpha-1}\,ds\, \phi'(t-x) dt-\frac{\phi(0)-\phi(x)}{\alpha x^{\alpha}}\\
         %&= \frac{1}\alpha\left\{-\int_x^\infty\frac{ \phi'(t-x)}{t^{\alpha}} dt+\frac{\phi(x)-\phi(0)}{ x^{\alpha}}\right\}.
      %\end{align*}
      The first calculation of the present proof shows, in particular,
      $$
        \frac{\alpha}{\Gamma(1-\alpha)} \int_0^x \frac{\phi(x)-\phi(x-s)}{s^{\alpha+1}}\,ds
        =
        \frac{1}{\Gamma(1-\alpha)} \left\{ \int_0^x \frac{\phi'(x-t)}{t^\alpha}\,dt - \frac{\phi(x)-\phi(0)}{x^\alpha}\right\}.
      $$
Since the reference measure $ds$ has no atom at $s=0$, there is no problem when considering $\phi'$ resp.\ $(\phi^0)'$ at $t=0$.
\end{proof}
The first part of the proof gives, in fact, the following alternative form of the Riemann--Liouville, Caputo, Marchaud fractional derivative with even extension and censored derivatives.
\begin{corollary}\label{cap-rld-con}
    Let $\alpha\in (0,1)$. Then
    \begin{align}
        \rld{0}{x}{\alpha}\phi
        &= \frac{\alpha}{\Gamma(1-\alpha)} \int_0^x \frac{\phi(x)-\phi(x-s)}{s^{\alpha+1}}\,ds + \frac{\phi(x)}{\Gamma(1-\alpha)x^\alpha};\label{mar-rli-kill}\\
        \cd{0}{x}{\alpha}\phi
        &= \frac{\alpha}{\Gamma(1-\alpha)} \int_0^x \frac{\phi(x)-\phi(x-s)}{s^{\alpha+1}}\,ds + \frac{\phi(x)-\phi(0)}{\Gamma(1-\alpha)x^\alpha};
        \\ \ed{0}{x}{\alpha}\phi
        &= \frac{\alpha}{\Gamma(1-\alpha)} \int_0^x \frac{\phi(x)-\phi(x-s)}{s^{\alpha+1}}\,ds+ \frac{\alpha}{\Gamma(1-\alpha)}\int_x^\infty \frac{\phi(x)-\phi(s-x)}{s^{\alpha+1}}\,ds ;
        \\ 
        \ced{0}{x}{\alpha}\phi
         &= \frac{\alpha}{\Gamma(1-\alpha)} \int_0^x \frac{\phi(x)-\phi(x-s)}{s^{\alpha+1}}\,ds.
    \end{align}
\end{corollary}
We note that there are other natural possibilities, e.g.\ odd extensions etc.
Now, let us consider a non--trivial example.%Let us define 
%\begin{equation}
%\tilde{u}(x)= \left\{
%\begin{aligned}
 %& u(x),\quad x\geq0 ,\\
%\\ & v(x), \quad x\leq 0.
%\end{aligned}
%\right.
%\end{equation}
\begin{example}
For a function $u:[0.\infty)\to \mathbb{R}$, the so--called censored derivative, defined  in Definition \ref{cen-fra-216} and Corollary \ref{cap-rld-con}, is 
\begin{align*}
\frac{\alpha}{\Gamma(1-\alpha)}\int_{0+}^{x}\frac{u(x)-u(x-t)}{t^{1+\alpha}}\,dt. 
\end{align*} 
If we compare it with the Marchaud derivative, we need an extension

\begin{equation}
\tilde{u}(x)= \left\{
\begin{aligned}
 & u(x),\quad x>0 ,\\
& u(0)=v(0), \quad x=0,
\\ & v(x), \quad x\leq 0.
\end{aligned}
\right.
\end{equation}
such that
\begin{align*}
 %&\frac{\alpha}{\Gamma(1-\alpha)}\int_{0+}^\infty \frac{u(x)-u(x-t)}{t^{1+\alpha}}\,dt
 %\\&=\underbrace{\frac{\alpha}{\Gamma(1-\alpha)}\int_{0+}^x \frac{u(x)-u(x-t)}{t^{1+\alpha}}\,dt}_{\mathrm{I}}+
% \underbrace{\frac{\alpha}{\Gamma(1-\alpha)}\int_x^\infty \frac{u(x)-u(x-t)}{t^{1+\alpha}}\,dt}_{\mathrm{II}}\
\frac{\alpha}{\Gamma(1-\alpha)}\int_x^\infty \frac{\tilde{u}(x)-\tilde{u}(x-t)}{t^{1+\alpha}}\,dt=\frac{\alpha}{\Gamma(1-\alpha)}\int_x^\infty \frac{u(x)-v(x-t)}{t^{1+\alpha}}\,dt=0.
\end{align*}

%Compared Part I with the Definition \ref{}, it is censored derivative. Now we want that Part II equals 0 when we use proper extension. 
%\begin{align*}
%\frac{\alpha}{\Gamma(1-\alpha)}\int_x^\infty \frac{u(x)-u(x-t)}{t^{1+\alpha}}\,dt=0
%\end{align*}
%We get 
%\begin{align*}
%\int_x^\infty u(x) t^{-(1+\alpha)}\,dt=\int_x^\infty u(x-t) t^{-(1+\alpha)}\,dt
%\end{align*}
By using  the change of variables $-y=x-t$, we obtain  $$ \frac{u(x)}{\alpha x^\alpha}= \int_0^\infty \frac{v(y)}{y^\alpha} \left(\frac1{(\frac{x}{y}+1)^{1+\alpha}}\right)\,\frac{dy}{y}.$$
By multiplying both sides with $\alpha x^\alpha$, we get   a Mellin convolution (defined by $\circledast$)
\begin{equation}\label{2-1-11}
u(x)=\alpha \int_0^\infty v(y) \left(\frac{(\frac{x}{y})^\alpha}{(\frac{x}{y}+1)^{1+\alpha}}\right)\,\frac{dy}{y}=:\alpha v\circledast g(x), 
\end{equation}
where $g(r):=r^\alpha(1+r)^{-\alpha-1}$.  

 Recall that the \textbf{Mellin transform} of a function $\phi:[0,\infty)\to\real$ is defined by $$\mathscr{M}(\phi; s)=\int_0^\infty x^{s-1}\phi(x)\,dx, \quad s\in \mathbb{C}.$$ If $x^{c-1} \phi(x)\in L^1(0,\infty)$ for all  $c\in (a,b)$ and some $a, b\in \real$, then $\mathscr{M}(\phi; s)$ exists and it is analytic in the strip $a<\Re s<b$. If $\phi$ is continuous and $t\mapsto\mathscr{M}(\phi; c+i t)$ is integrable in $(-\infty, \infty)$, then we have an inverse transform,
$$\phi(x)=\frac1{2i\pi}\int_{c-i\infty}^{c+i\infty}x^{-z}\mathscr{M}(\phi; z)\,dz.$$
Taking the Mellin transform on both sides of eq. \eqref{2-1-11} we have 
$$\mathscr{M}u(z)=\alpha \mathscr{M}v(z)\mathscr{M}g(z),$$
for all $z\in \mathbb{C}$, where the Mellin transforms of $v$ and $g$ are well-defined.
Calculating the Mellin transform of $g$, we obtain that
\begin{align*}
\mathscr{M}g(z)=B(1-z,\alpha+z)\textrm{ for all }-\alpha<\operatorname{Re}(z)<1.
\end{align*}
Formally, we obtain that $v$ is given by

$$v(x)=\mathscr{M}^{-1}\left(\frac{\mathscr{M}u}{\alpha \mathscr{M}g}\right)(x)=\mathscr{M}^{-1}\left(\frac{\mathscr{M}u}{\alpha B(1-\cdot,\alpha+\cdot)}\right)(x).$$
If we assume for some $x\in (-\alpha,1)$,   that   $$y\mapsto \frac{\mathscr{M}u(x+i y)}{\alpha B(1-x-i y,\alpha+x+i y)}  \in L^1(\real),$$  then the above inverse Mellin transform is well defined i.e. 
\begin{equation}
v(x)=\frac1{2\pi i}\int_{c-i\infty}^{c+i\infty} x^{-z} \frac{\mathscr{M}u(z)}{\alpha B(1-z,\alpha+z)}\,dz<\infty
\end{equation}
with $-\alpha<c<1$.  This gives us an idea of how to extend the Marchaud fractional derivative to get  other fractional derivatives. 
\end{example}

\subsection{Interpolation theorem and fractional Sobolev space}

In this section we will introduce an important interpolation theorem which we will use in the later.  For further information see Triebel \cite{1985_Triebel, 1994_Triebel} , Bennett $\&$  Sharpeley
\cite{1988_Bennett} and Reed $\&$ Simon \cite{1980_Reed}. 
\begin{definition} Let $\mathcal{X}$ be a complex vector space. Two norms $\|\cdot\|_{\mathcal{X}_0}, \|\cdot\|_{\mathcal{X}_1}$ on $\mathcal{X}$ are called \textbf{consistent} if any sequence $\{x_n\}$ that converges to zero in one norm, and which is Cauchy in other norm, converges to zero in both norms. If  $\|\cdot\|_{\mathcal{X}_0}$  and $\|\cdot\|_{\mathcal{X}_1}$ are consistent, we define  $\|x\|_{\mathcal{X}_0+ \mathcal{X}_1}=\inf\{\|x_0\|_{\mathcal{X}_0}+\|x_1\|_{\mathcal{X}_1}:x=x_0+x_1, x_j\in \mathcal{X}, j=0,1\}$
\end{definition}
 We denote by $\mathcal{X}_0$, $\mathcal{X}_1$ and $\mathcal{X}_0+\mathcal{X}_1$ the completion of $\mathcal{X}$ with respect to the norms $\|\cdot\|_{\mathcal{X}_0}$, $\|\cdot\|_{\mathcal{X}_1}$ and $\|\cdot\|_{\mathcal{X}_0+ \mathcal{X}_1}$, respectively.

Assume $S=\{z\in \mathbb{C}: 0< \Re z<1\}$ be an open strip in the complex plane.   Assume $\|\cdot\|_{\mathcal{X}_0}$ and $\|\cdot\|_{\mathcal{X}_1}$ be two consistent norms on a complex vector space $\mathcal{X}$.   Denote $F[\mathcal {X}]$ to be the space of continuous functions $\phi(z)$, from $\overline{S}$ to $\mathcal{X}_0+ \mathcal{X}_1$ which are analytic in $S$, with the following properties:
\begin{enumerate}[(1)]
\item $\sup_{z\in \overline{S}}\|\phi(z)\|_{\mathcal{X}_0+\mathcal{X}_1}<\infty.$
\item $\phi(iy)\in \mathcal{X}_0$ and $\phi(1+iy)\in \mathcal{X}_1$ with $y\in \real$, are continuous with respect to the norm $\|\cdot\|_{ \mathcal{X}_0}$ and $\|\cdot\|_{\mathcal{X}_1}$, respectively:  $$\|\phi\|_{F[\mathcal {X}]}=\sup_{y\in \real }\{\|\phi(iy)\|_{\mathcal{X}_0}, \quad \|\phi(1+iy)\|_{\mathcal{X}_1}\}<\infty.$$
\end{enumerate}
$F[\mathcal {X}]$ with the norm $\|\cdot\|_{F[\mathcal {X}]}$ is a Banach space, and for each $t\in [0,1]$ the subspace $K_t=\{\phi\in F[\mathcal {X}]\,|\,\phi(t)=0\}$ is $\|\cdot\|_{F[\mathcal {X}]}$-closed. 
Define $\mathcal {\tilde{X}}_t=F[\mathcal {X}]/K_t, 0\leq t\leq 1$, and the quotient norm on $\mathcal {\tilde{X}}_t$ is $\|\cdot\|_{\mathcal {\tilde{X}}_t}$. Note that $\mathcal {X}$ may be identified with a subset of $\mathcal {\tilde{X}}_t$ which takes each $x\in\mathcal {X}$ into $[x]$, the equivalence class of the constant function whose value is $x$. Further, $\mathcal {\tilde{X}}_t$ may be identified with a subset of $\mathcal{X}_0+ \mathcal{X}_1$ under the map which takes an equivalence class $[\phi]$ into the common value of its members at $t$. This map  is clearly injective, and the following computation shows that it is continuous: Let $[\phi] \in\mathcal {\tilde{X}}_t $, $x=\phi(t)$ and $ \|x\|_{\mathcal{X}_0+ \mathcal{X}_1}\leq \|\phi\|_{F[\mathcal {X}]}$. Thus $$\|x\|_{\mathcal{X}_0+ \mathcal{X}_1}\leq \|[\phi]\|_{\mathcal {\tilde{X}}_t}=\inf\{\|\phi\|_{F[\mathcal {X}]}: \phi\in F[\mathcal {X}], \, \phi(t)=x\}.$$We denote by $\mathcal {X}_t$ is the completion of $\mathcal {X}$ with respect to the norm $\|\cdot\|_{\mathcal {\tilde{X}}_t}$.  As usual, we set $$\|x\|_{\mathcal {X}_t}=\|[\phi]\|_{\mathcal {\tilde{X}}_t},$$ where $\phi\in F[\mathcal {X}]$  such that $\phi(t)=x$,  see Reed $\&$ Simon \cite{1975_Reed}.

The next theorem is  the Calder\'{o}n Lions Interpolation Theorem, which is  taken from Reed $\&$ Simon \cite{1975_Reed}.  
\begin{theorem}\label{clit-11} Let $\mathcal{X}$ and $\mathcal{Y}$ be complex vector spaces with given consistent norms\linebreak $\|\cdot\|_{\mathcal{X}_0}, \|\cdot\|_{\mathcal{X}_1}$ on $\mathcal{X}$ and $\|\cdot\|_{\mathcal{Y}_0}, \|\cdot\|_{\mathcal{Y}_1}$ on $\mathcal{Y}$. Suppose $\Lambda(\cdot)$ is an analytic, uniformly bounded, continuous and  $L(\mathcal{X}_0+ \mathcal{X}_1,  \mathcal{Y}_0+ \mathcal{Y}_1)$-valued function on the strip $\overline{S}$ with the following properties:
\begin{enumerate}[(1)]
\item $\Lambda(t) :\mathcal{X} \to \mathcal{Y}$ for each $t\in (0,1)$,
\item for all $y\in \real$, $\Lambda(iy):\mathcal{X}_0\to \mathcal{Y}_0$ and $M_0=\sup_{y\in \real}\|\Lambda(iy)\|_{L(\mathcal{X}_0, \mathcal{Y}_0)}<\infty$,
\item for all $y\in \real$, $\Lambda(1+iy):\mathcal{X}_1\to \mathcal{Y}_1$ and $M_1=\sup_{y\in \real}\|\Lambda(1+iy)\|_{L(\mathcal{X}_1,  \mathcal{Y}_1)}<\infty$,
\end{enumerate}
then $T(t)[\mathcal {X}_t]\subset \mathcal {Y}_t$ and 
$\|T(t)\|_{L(\mathcal {X}_t, \mathcal {Y}_t)}\leq M_0^{1-t}M_1^t$.
\end{theorem}
To illustrate the above interpolation theorem, we will provide an important example as application of Theorem \ref{clit-11}. 
\begin{example}\label{exa-12}
Let $\Omega\subseteq{\real}$, i.e. $\Omega$ can take $\real$. If we take $\mathcal{X}=L^1(\Omega)\cap L^\infty(\Omega)$  and set $\|\cdot\|_{\mathcal{X}_0}:=\|\cdot\|_{L^1(\Omega)}$ and $\|\cdot\|_{\mathcal{X}_1}:=\|\cdot\|_{L^\infty(\Omega)}$, we obtain that $\mathcal{X}_t=L^{t^{-1}}(\Omega)$ for $t\in (0,1)$. For a detailed discussion see Reed and Simon \cite[P.38]{1975_Reed}.
\end{example}

  We will introduce the definition of fractional weighted Sobolev space $W^{p,\beta}(\Omega, w)$ and  fractional Sobolev space $W^{p,\beta}(\Omega)$, where $1<p<\infty$ and $0<\beta\leq 1$. For further properties, we refer to Leoni \cite{2023_Leoni}.
  \begin{definition}Let $\Omega\subseteq{\real}$, i.e. $\Omega$ can take $\real$ and $\alpha\in (0,1)$.  Assume $w$ is a bounded, positive weight over the domain $\Omega$. Define the following weight set $$W(\alpha)=\left\{w: \Omega \to [0, \infty), \quad \|w\|_\infty<\infty\quad  \text{and} \quad \sup_{x\in \Omega}\left|\int_1^\infty \frac{w(x)}{y^{1+\alpha} w(x-y)}\,dy\right|<\infty \right\}.$$
  \end{definition}
 The set of $W(\alpha)$ is not empty. For example, if $\gamma<\alpha$ and  $w(y)=\frac{1}{1+|y|^\gamma}$, $w(y)\in W(\alpha)$.
\begin{definition}\label{sob-0}
Let $\Omega\subseteq{\real}$, i.e. $\Omega$ can take $\real$ and   $w\in W(\alpha)$.
 The fractional weight Sobolev space $W^{p,\beta}(\Omega, w)$, $p\in [1,\infty)$, $\beta\in (0,1]$, consists of all $\phi\in L^p(\Omega, w)$ such that the norm
$$
\|\phi\|_{ W^{p,\beta} (\Omega,w)}^p =\|\phi\|_{L^p(\Omega, w)}^p +\int_\Omega\int_\Omega\frac{|\phi(x)-\phi(y)|^p}{|x-y|^{p\beta+1}}\,dx\,dy,
$$
% The norm is defined by $\phi\in W^{p,s}$\label{fra-sob-spa}, $\|\phi\|_{ W^{p,s}}=\|\phi\|_{L^p}+(\int_0^T\int_0^T\frac{|\phi(x)-\phi(y)|^p}{|x-y|^{ps+1}}dxdy)^{\frac{1}{p}}$.
is finite. If $p=\infty$, we have 
$$\|\phi\|_{ W^{\infty,\beta}(\Omega,w)}=\|\phi\|_{L^\infty(\Omega, w)}+\operatorname*{esssup}_{(x,y)\in \Omega\times \Omega \setminus{\{x=y\}} }\frac{|\phi(x)-\phi(y)|}{|x-y|^\beta}.$$ 
%Let $\Omega\subseteq{\real}$, i.e. $\Omega$ can take $\real$. 
Attention: this is not a standard fractional weighted Sobolev space, for the standard defintion we refer to \cite{1985_Triebel,1994_Triebel}.

If we take $w=1$, we have the classical  fractional Sobolev sapce $W^{p,\beta}(\Omega)$, $p\in [1,\infty)$, $\beta\in (0,1]$, consists of all $\phi\in L^p(\Omega)$ such that the norm
$$
\|\phi\|_{ W^{p,\beta}}^p =\|\phi\|_{L^p}^p +\int_\Omega\int_\Omega\frac{|\phi(x)-\phi(y)|^p}{|x-y|^{p\beta+1}}\,dx\,dy,
$$
% The norm is defined by $\phi\in W^{p,s}$\label{fra-sob-spa}, $\|\phi\|_{ W^{p,s}}=\|\phi\|_{L^p}+(\int_0^T\int_0^T\frac{|\phi(x)-\phi(y)|^p}{|x-y|^{ps+1}}dxdy)^{\frac{1}{p}}$.
is finite. If $p=\infty$, we have 
$$\|\phi\|_{ W^{\infty,\beta}}=\|\phi\|_\infty+\operatorname*{esssup}_{(x,y)\in \Omega\times \Omega \setminus{\{x=y\}} }\frac{|\phi(x)-\phi(y)|}{|x-y|^\beta}.$$ 
\end{definition}
%\begin{definition}\label{sob-1} 
%\end{definition}
\section{Main results}
In general, when we take the derivative of a function, we will lose regularity. In fact, there are similar results for the Marchaud fractional derivative, indicating a potential loss of regularity. %For instance, we will show this in the following theorem: If our function is in the $\alpha$-order H\"older space denoted by $C^\alpha$, we can obtain a continuous function. 

\begin{theorem}
 Let $\beta>0, \,w\in W(\alpha), \,  \phi\in C_b^\beta (\real, w), \, 0<\alpha\leq\beta\leq1 $.  Then $\mad{}{+}{\alpha}\phi(\cdot)\in C_b^{\beta-\alpha}(\real, w).$
\end{theorem}
\begin{proof}
 We define the space $C_b^\beta (\real, w)$ equipped with the norm  $$\|\phi\|_{C_b^{\beta}(\real, w)}=\|\phi w\|_{\infty}+\operatorname*{sup}_{(x,y)\in \Omega\times \Omega \setminus{\{x=y\}} }\frac{|\phi(x)-\phi(y)|}{|x-y|^\beta}.$$
 In view of \eqref{mar-def+}, it is sufficient to show 
    $$\mad{}{+}{\alpha}\phi(x)\in C_b^{\beta-\alpha}(\real, w).$$
First we show boundedness:
\begin{align*}
&\sup_{x\in \real} \left|\mad{}{+}{\alpha}\phi(x) w(x)\right|\\&=\sup_{x\in \real} \left|w(x)\int_{0}^\infty \frac{\phi(x)-\phi(x-t)}{t^{1+\alpha}}\,dt\right|
\\&\leq \sup_{x\in \real} |w(x)|\sup_{x\in \real} \left|\int_{0}^1\frac{\phi(x)-\phi(x-t)}{t^{1+\alpha}}\,dt\right|+ \sup_{x\in \real}\left|w(x)\int_{1}^\infty \frac{\phi(x)-\phi(x-t)}{t^{1+\alpha}}\,dt\right|
\\&\leq C\sup_{x\in \real} \left\{\left|\int_{0}^1\frac{\phi(x)-\phi(x-t)}{t^{1+\alpha}}\,dt\right|\right\}+\sup_{x\in \real}\left\{\left|w(x)\int_{1}^\infty \frac{\phi(x)-\phi(x-t)}{t^{1+\alpha}}\,dt\right|\right\}
\\&\leq C_1\sup_{x\in \real} \left\{\left|\int_{0}^1\frac{t^\beta}{t^{1+\alpha}}\,dt\right|+\left|\int_{1}^\infty \frac{1}{t^{1+\alpha}}\,dt\right|+\left|\int_1^\infty \frac{w(x)}{t^{1+\alpha} w(x-t)}\,dt\right|\right\}
<\infty.
\end{align*}
Secondly, we have 
\begin{align*}
   &\frac{\Gamma(1-\alpha)}{\alpha}\left|\mad{}{+}{\alpha}\phi(x+h)-\mad{}{+}{\alpha}\phi(x)\right|\\&= \left|\int_{0}^\infty \frac{\phi(x+h)-\phi(x+h-t)}{t^{1+\alpha}}\,dt-\int_{0}^\infty \frac{\phi(x)-\phi(x-t)}{t^{1+\alpha}}\,dt\right|
   \\&= \left|\int_{-h}^\infty \frac{\phi(x+h)-\phi(x-t)}{(t+h)^{1+\alpha}}\,dt-\int_{0}^\infty \frac{\phi(x)-\phi(x-t)}{t^{1+\alpha}}\,dt\right|
    \\&= \left|\int_{0}^\infty \frac{\phi(x+h)-\phi(x-t)}{(t+h)^{1+\alpha}}\,dt-\int_{0}^\infty \frac{\phi(x)-\phi(x-t)}{t^{1+\alpha}}\,dt+\int_{-h}^0 \frac{\phi(x+h)-\phi(x-t)}{(t+h)^{1+\alpha}}\,dt\right|
    \\ &\leq  \left|\int_{0}^\infty \frac{\phi(x+h)-\phi(x-t)}{(t+h)^{1+\alpha}}\,dt-\int_{0}^\infty \frac{\phi(x)-\phi(x-t)}{t^{1+\alpha}}\,dt\right|+\left|\int_{-h}^0 \frac{\phi(x+h)-\phi(x-t)}{(t+h)^{1+\alpha}}\,dt\right|
    \\ &= \left|\int_{0}^\infty \frac{\phi(x+h)-\phi(x-t)}{(t+h)^{1+\alpha}}+\frac{\phi(x)}{(t+h)^{1+\alpha}}-\frac{\phi(x)}{(t+h)^{1+\alpha}}+\frac{\phi(x)-\phi(x-t)}{t^{1+\alpha}}\,dt\right|+\text{I}
     \\ &\leq \left|\int_{0}^\infty \frac{\phi(x)-\phi(x-t)}{(t+h)^{1+\alpha}-t^{1+\alpha}}\right|\,dt+\int_{0}^\infty\left|\frac{\phi(x+h)-\phi(x)}{(t+h)^{1+\alpha}}\right|\,dt+\text{I}
    \\&=\text{III}+\text{II}+\text{I}
\end{align*}
It is easy to see that 
$$\text{III}\leq M \int_{0}^\infty t^\beta \left|(t+h)^{-1-\alpha}-t^{-1-\alpha}\right|\,dt\leq M h^{\beta-\alpha},$$
and $$\text{II}\leq M h^\beta \int_{0}^\infty (t+h)^{-1-\alpha}\,dt \leq M h^{\beta-\alpha}.$$
%$$|\text{II}|\leq M \int_{0}^\infty (t+h)^{-1-\alpha}\,dt \leq M h^{\beta-\alpha}.$$
Hence, we have $\|\mad{}{+}{\alpha}\phi\|_{C^{\beta-\alpha}}\leq M\|\phi\|_{C^\beta}$.
\end{proof}
%If we take $w=1$, we have the following Corollary. 
For  $w=1$, we obtain the following result for $C^s(\real)$, $s\in (0,1)$.  Applying the $\alpha$--order Marchaud fractional derivative to a function in $C^s(\real)$, we will lose $\alpha$ orders of smoothness. 
\begin{corollary} \label{rie-lil-7} 
    Let $\beta>0, \,  \phi\in C^\beta (\real), \, 0<\alpha<\beta\leq1 $.  Then $\mad{}{+}{\alpha}\phi(\cdot)\in C^{\beta-\alpha}(\real).$
\end{corollary}

When $p\geq 1$, the $\alpha$ Marchaud fractional derivatives maps $W^{1, s}(\real, w)$ to  $W^{1,s-\alpha}(\real, w)$ under certain conditions. 

\begin{theorem} 
%Define 
%$w:\real\to \real$, given by 
%\begin{align*}
%    w(x)=\left\{\begin{aligned}
%        &1, \quad x>0,\\
%        &\frac{1}{1+x^2}, \quad x\leq 0.
%    \end{aligned}
%    \right.
%\end{align*} 
Let $w\in W(\alpha)$ and $\phi\in W^{1,s}(\real, w),\, 0<\alpha <s<1$, then  we have  $\mad{}{+}{\alpha}\phi\in W^{1, s-\alpha}(\real, w)$.
\end{theorem}
\begin{proof}
    According to the definition of $W^{1,s-\alpha}(\real, w)$, we have the following expression
\begin{align*}
\|\mad{}{+}{\alpha}\phi\|_{W^{1,s-\alpha}}&=\|\mad{}{+}{\alpha}\phi\|_{L^1(\real, w)}+\left(\int_\real\int_\real\frac{|\mad{}{+}{\alpha}\phi(x)-\mad{}{+}{\alpha}\phi(y)|}{|x-y|^{(s-\alpha)+1}}\,dx\,dy\right)
\\&=\text{I $+$ II}.
\end{align*}
For the term I we have,
     \begin{align*}
         &\|\mad{}{+}{\alpha}\phi\|_{L^1(\real, w)}\\&=\int_\real\int_{-\infty}^x\frac{|\phi(x)-\phi(y)|}{|x-y|^{1+\alpha}} \,dy \,w(x)\,dx
         \\&\leq  \int_\real \int_\real \frac{|\phi(x)-\phi(y)|}{|x-y|^{1+\alpha}} \,dy  \,w(x)\,dx
         %\\&=\int_\real \int_\real \frac{|\phi(x)-\phi(y)|}{|x-y|^{1+\alpha}}\,dx\,dy
         \\&\leq  \int_\real \int_{|x-y|\geq1} \frac{|\phi(x)-\phi(y)|}{|x-y|^{1+\alpha}}\,dy  \,w(x)\,dx+\sup_{x\in \real} w(x) \int_\real \int_{|x-y|\leq1} \frac{|\phi(x)-\phi(y)|}{|x-y|^{1+\alpha}}\,dy \,dx
          %\\&\leq\int_\real \int_{|x-y|\geq1} \frac{|\phi(x)|+|\phi(y)|}{|x-y|^{1+\alpha}}\,dx\,dy+\int_\real \int_{|x-y|\leq1} \frac{|\phi(x)-\phi(y)|}{|x-y|^{1+\alpha}}\,dx\,dy
          \\&\leq 2\int_\real \int_{|z|\geq1} \frac{|\phi(x)|}{|z|^{1+\alpha}}\,dz \, w(x)\,dx+\sup_{x\in \real} w(x) \int_\real \int_{|x-y|\leq1} \frac{|\phi(x)-\phi(y)|}{|x-y|^{1+\alpha}}\,dy \,dx
           \\&\leq 2  M \|\phi\|_{L^1(\real, w)} +\sup_{x\in \real} w(x) \int_\real \int_{|x-y|\leq1} \frac{|\phi(x)-\phi(y)|}{|x-y|^{1+s}}\,dy \,dx
           \\&\leq 2  M \|\phi\|_{L^1(\real, w)} + \|w\|_\infty\int_\real \int_\real \frac{|\phi(x)-\phi(y)|}{|x-y|^{1+s}}\,dy \,dx
           \\&\leq C \|\phi\|_{W^{1,s}(\real, w)},
     \end{align*}
    for the last third inequality, we use the integrability of $\int_{|z|\geq1}\frac{1}{|z|^{1+\alpha}}\,dz$ and the monotonicity of $|x-y|^{-1-\alpha}$.
     
     For the term II.
without loss of generality, we assume $x\geq y$ and use the symmetry. We have the following estimate, 
\begin{align*}
& \int_\real\int_\real \frac{\bigg|\int_{-\infty}^x\frac{\phi(x)-\phi(t)}{(x-t)^{1+\alpha}}\,dt-\int_{-\infty}^y\frac{\phi(y)-\phi(t)}{(y-t)^{\alpha+1}}\,dt\bigg|}{|x-y|^{(s-\alpha)+1}} \, dx\,dy
\\&= \int_\real\int_\real \frac{\bigg|\int_{-\infty}^x\frac{\phi(x)-\phi(t)}{(x-t)^{1+\alpha}}\,dt-\int_{-\infty}^y\frac{\phi(y)-\phi(t)}{(y-t)^{\alpha+1}}\,dt\bigg|}{|x-y|^{(s-\alpha)+1}} \, dx\,dy
\\&\leq \int_\real\int_\real \frac{\bigg|\int_{-\infty}^y \frac{\phi(x)-\phi(t)}{(x-t)^{1+\alpha}}\,dt-\int_{-\infty}^y\frac{\phi(y)-\phi(t)}{(y-t)^{\alpha+1}}\,dt\bigg|}{|x-y|^{(s-\alpha)+1}} \, dx\,dy + \int_\real\int_\real \frac{\bigg|\int_y^x \frac{\phi(x)-\phi(t)}{(x-t)^{1+\alpha}}\,dt\bigg|}{|x-y|^{(s-\alpha)+1}} \, dx\,dy
\\&=\text{III+IV}.
\end{align*}
For term IV, we can just use Fubini's theorem. In fact, 
\begin{align*}
&\int_\real\int_\real \frac{\bigg|\int_y^x \frac{\phi(x)-\phi(t)}{(x-t)^{1+\alpha}}\,dt\bigg|}{|x-y|^{(s-\alpha)+1}} \, dx\,dy%\\&\leq\int_\real\int_{-\infty}^x \frac{\int_y^x \frac{|\phi(x)-\phi(t)|}{|x-t|^{1+\alpha}}\,dt}{|x-y|^{(s-\alpha)+1}} \, dx\,dy
\\&=\int_\real\int_{\infty}^x \int_{-\infty}^t \frac{ |\phi(x)-\phi(t)||x-t|^{-1-\alpha}}{|x-y|^{(s-\alpha)+1}} \,dy \,dt\, dx
\\&\leq  M\int_\real\int_\real  \frac{|\phi(x)-\phi(t)|}{|x-t|^{1+s}} \, dx\,dt
\\&\leq  M\left(\int_\real\int_\real  \frac{|\phi(x)-\phi(t)|}{|x-t|^{1+s}} \, dx\,dt+\|\phi\|_{L^1(\real, w)}\right)
\\&= M\|\phi\|_{W^{1, s}(\real, w)}
\end{align*}
In the penultimate inequality estimate we calculate the integral exactly and use $|x|^{\alpha-s}\leq |x-t|^{\alpha-s}$ 
to get the above result.
\\For part III, we will get 
\begin{align*}
 &\int_\real\int_\real\frac{\bigg|\int_{-\infty}^y \frac{\phi(x)-\phi(t)}{(x-t)^{1+\alpha}}\,dt-\int_{-\infty}^y\frac{\phi(y)-\phi(t)}{(y-t)^{\alpha+1}}\,dt\bigg|}{|x-y|^{(s-\alpha)+1}} \,dy \, dx
\\&\leq  \int_\real\int_\real \frac{ |\phi(x)-\phi(y)| \int_{-\infty}^y (x-t)^{-1-\alpha}\,dt}{|x-y|^{(s-\alpha)+1}}  \,dy \, dx\\&\quad+ \int_\real\int_\real \frac{\int_{-\infty}^y|\phi(y)-\phi(t)||(x-t)^{-\alpha-1}-(y-t)^{-\alpha-1}|\,dt}{|x-y|^{(s-\alpha)+1}} \,dy \, dx
\\&\leq M \left(\int_\real\int_\real \frac{ |\phi(x)-\phi(y)| }{|x-y|^{s+1}} \,dy \, dx+\|\phi\|_{L^1(\real, w)}\right)\\&\quad+ \int_\real\int_\real\frac{\int_{-\infty}^y|\phi(y)-\phi(t)||(x-t)^{-\alpha-1}-(y-t)^{-\alpha-1}|\,dt}{|x-y|^{(s-\alpha)+1}}\,dy \, dx
\\&= M\|\phi\|_{W^{1,s}(\real, w)}+ \int_\real\int_\real\frac{\int_{-\infty}^y|\phi(y)-\phi(t)||(x-t)^{-\alpha-1}-(y-t)^{-\alpha-1}|\,dt}{|x-y|^{(s-\alpha)+1}} \,dy \, dx.
%\\&=\text{VII+VIII}
\end{align*}
Next we deal with the the right hand side, by using the monotonicity of $x^\alpha$ to get
\begin{align*}
&\int_\real\int_\real \frac{\int_{-\infty}^y|\phi(y)-\phi(t)||(x-t)^{-\alpha-1}-(y-t)^{-\alpha-1}|\,dt}{|x-y|^{(s-\alpha)+1}} \,dy \, dx
\\&\leq C_\alpha \int_\real\int_\real \frac{\int_{-\infty}^y|\phi(y)-\phi(t)|\frac{(x-t)^\alpha |x-y|}{(x-t)^{\alpha+1}(y-t)^{\alpha+1}}|\,dt}{|x-y|^{(s-\alpha)+1}}\,dy \, dx
\\&=  C_\alpha \int_\real\int_\real \int_{-\infty}^y \frac{|\phi(y)-\phi(t)|}{|y-t|^{s+1}} \frac{|y-t|^{s+1}}{|x-y|^{s-\alpha}|x-t||y-t|^{\alpha+1}} \,dt\,dy \, dx
\\&=C_\alpha \int_\real\int_t^\infty \int_y^\infty \frac{|\phi(y)-\phi(t)|}{|y-t|^{s+1}} \frac{|y-t|^{s+1}}{|x-y|^{s-\alpha}|x-t||y-t|^{\alpha+1}} \, dx\,dy\,dt
\\&\leq C_\alpha M \int_\real\int_{\real}  \frac{|\phi(y)-\phi(t)|}{|y-t|^{s+1}}  \,dy\,dt.
\end{align*}
Finally we show that  $\int_y^\infty  \frac{|y-t|^{s+1}}{|x-y|^{s-\alpha}|x-t||y-t|^{\alpha+1}} \, dx$ is bounded. In fact, after a change of variable we have the following estimate,
\begin{align*}
&\int_y^\infty  \frac{|y-t|^{s+1}}{|x-y|^{s-\alpha}|x-t||y-t|^{\alpha+1}} \, dx
\\&=\int_0^\infty\frac{1}{|\frac{u}{y-t}|^{s-\alpha}} \frac{1}{|1+\frac u{y-t}|}\frac {du}{|y-t|}
\\&= \int_0^\infty\frac{1}{|v|^{s-\alpha}} \frac{1}{|1+v|}dv
%\\&\leq \int_0^\infty \frac{1}{|v|^{s-\alpha}} \frac{1}{|1+v|}dv
\\&<M.
\end{align*}
Since $M<\infty$, so we have $\|\mad{}{+}{\alpha}\phi\|_{W^{1,s-\alpha}(\real, w)} \leq C \|\phi\|_{W^{1,s}(\real, w)}$.
\end{proof}
For $p=1$ and $w=1$, we obtain the following result for $W^{1,s}(\real)$, $s\in (0,1)$.  Applying the $\alpha$ order Marchaud fractional derivative to a function in $W^{1,s}(\real)$, we will lose $\alpha$ orders of smoothness, for the direct proof, see \cite{2023_Li}. 
\begin{corollary} \label{map-frc-2.2.9} Let $ \phi\in W^{1,s}(\real), 0<\alpha <s<1$, then  we have  $\mad{}{+}{\alpha}\phi\in W^{1, s-\alpha}(\real)$.
\end{corollary}
Now, we use the interpolation theorem to show:  When $p\geq 1$, the $\alpha$ Marchaud fractional derivatives maps weight fractional Sobolev space $W^{p, s}(\real, w)$ to  $W^{p,s-\alpha}(\real, w)$ under certain kind of conditions. 
\begin{theorem}\label{mp-3-5} Let $w\in W(\alpha)$ and  $ \phi\in W^{p,s}(\real, w)$,  $0<\alpha < s<1\leq p$, then  we have  \linebreak $\mad{}{+}{\alpha}\phi\in W^{p,s-\alpha}(\real, w)$.
\end{theorem}
\begin{proof}

 We want to use the  Calder\'on--Lions interpolation theorem, cf. Theorem \ref{clit-11}.  Using the notation from Theorem \ref{clit-11}, we have $$\mathcal{X}=C_c^\infty(\real), \quad \mathcal{Y}=W^{1, s-\alpha}(\real, w)\cap W^{\infty, s-\alpha}(\real, w)$$
and $$\|\cdot\|_{\mathcal{X}_1}=\|\cdot\|_{W^{1,s}}, \quad \|\cdot\|_{\mathcal{X}_0}=\|\cdot\|_{C^s},$$ $$\|\cdot\|_{\mathcal{Y}_1}=\|\cdot\|_{W^{1, s-\alpha}}, \quad \|\cdot\|_{\mathcal{Y}_0}=\|\cdot\|_{W^{\infty, s-\alpha}}.$$
Denote $C_0^s (\real, w)=\overline{\mathcal{X}}^{\|\cdot\|_{C_b^s(\real, w)}}\subset C_b^s (\real, w)$.  Moreover  $T(t)=\mad{}{+}{\alpha}$ (independent of $t$).

We know $\mad{}{+}{\alpha}: C_0^s (\real, w)\to C_b^{s-\alpha}(\real, w)$,  as well as $\mad{}{+}{\alpha}:W^{1, s}(\real, w)\to W^{1, s-\alpha}(\real, w)$ see  Theorem \ref{rie-lil-7} and \ref{map-frc-2.2.9}, respectively. %We also have $W^{1,s}(\real)=W_0^{1,s}(\real)$ using Theorem \ref{sob-tes}, and by the definition of the respective norm,  
We also have see Definition \ref{sob-0}, \linebreak $ C_b^{s-\alpha}(\real,w)=W^{\infty, s-\alpha}(\real, w)$ and $C_b^s (\real,w)=W^{\infty, s}(\real,w)$.  Therefore, all three conditions  of Theorem \ref{clit-11} are satisfied,  and we get that $\mad{}{+}{\alpha}: \mathcal{X}_t\to \mathcal{Y}_t$, $t\in [0,1]$. It remains to identify the interpolation spaces.  We will see that 
$$ \mathcal{X}_t=W^{1/t,s}(\real, w), \quad \text{and} \quad \mathcal{Y}_t=W^{1/t, s-\alpha}(\real, w).$$
In order to identify $\mathcal{X}_t$ we remark that we can identify $W^{1/t,s}(\real, w)$ with a closed subspace of 
$$L^{1/t}\left(\real, w(x)dx\right)\times L^{1/t}\left(\real\times\real\setminus{\{x=y\}}, \quad \frac{dx\,dy}{|x-y|}\right),$$
which is equipped with the norm 
$$\|(g,h)\|_{1/t}^t=\|g\|_{L^{1/t}(w(x)dx)}^t+\|h\|_{L^{1/t}(dm)}^t,$$
 where $dm=dx\,dy/|x-y|$.

\noindent Therefore, it will be enough to show that the norms $\|\cdot\|_{\mathcal{X}_t}$ and $\|\cdot, \cdot\|_{1/t}$ are equivalent on the space
$$\hat{\mathcal{X}}=\left\{\left(\phi(u), \frac{\phi(x)-\phi(y)}{|x-y|^s}\right):\phi\in C_c^\infty(\real)\right\}.$$ 

\noindent Proof of $W^{1/t,s}(\real, w)\subset\mathcal{X}_t$. Recall from the discussion preceding Theorem \ref{clit-11} that $$\|\phi\|_{\mathcal{X}_t}=\inf\{\|\tilde{\phi}\|_{F[\mathcal{X}]}: \tilde{\phi}\in F[\mathcal{X}], \quad \tilde{\phi}(t)=\phi\},$$
where $F[\mathcal{X}]$ was the family of analytic maps connecting $\mathcal{X}_0$ and $\mathcal{X}_1$, see before Theorem \ref{clit-11}.

Pick any $\phi\in \mathcal{X}=C_c^\infty(\real)$ such that $\|\phi\|_{W^{1/t,s}(\real, w)}=1$ and define  for $z\in \bar{S}$, where $S=\{z\in \mathbb{C} : 0<\Im z<1\}$  an analytic function 
\begin{align*}
    g(z)(u,x,y) &= 
        \begin{pmatrix}
          |\phi(u)|^{z/t}\exp(i\arg\phi(u)) \\           
          \left|\frac{\phi(x)-\phi(y)}{|x-y|^s}\right|^{z/t}\exp\left(i\arg\left(\frac{\phi(x)-\phi(y)}{|x-y|^s}\right)\right)
    \end{pmatrix}
  \end{align*}
Note that we  can identify $\mathcal{X}$ and $\hat{\mathcal{X}}$.
Clearly, $g(z)\in \hat{\mathcal{X}}$ and for all $v\in (0, 1)$
$$\|g(iv)\|_{\infty}=2,$$ and \begin{align*}
\|g(1+iv)\|_1&=\int_\real  ||\phi(u)|^{(1+iv)/t}| \, w(u)du+\int_\real\int_\real\left|\left|\frac{\phi(x)-\phi(y)}{|x-y|^s}\right|^{(1+iv)/t}\right|\frac{\,dx\,dy}{|x-y|}
\\&=\|\phi\|_{W^{ 1/t,s}(\real, w)}\\&=1.
\end{align*}
Thus, by the three lines theorem 
$\|g\|_{F[\mathcal{X}]}\leq 2=2\|\phi\|_{W^{ 1/t,s}(\real, w)}$ and, using the definition of $\|\cdot\|_{\mathcal{X}_t}$ we get  $\|\phi\|_{\mathcal{X}_t}\leq 2\|\phi\|_{W^{ 1/t,s}(\real, w)}$ for all $\phi\in\mathcal{X}=C_c^\infty(\real)$. This proves that  $W^{1/t, s}(\real,w)$ is a subset of $\mathcal{X}_t$, i.e. $W^{1/t, s}(\real, w)\subset\mathcal{X}_t$.

In oder to see the converse in conclusion, 
$W^{1/t,s}(\real, w)\supset\mathcal{X}_t$, we take any map 
\begin{align*}
    \psi(z)(u,x,y) &= 
        \begin{pmatrix}
         \psi_1(z)(u) \\           
          \psi_2(z)(x,y))
    \end{pmatrix} \textrm{ $\in \hat{\mathcal{X}}$ }
  \end{align*}
with $\psi(z)\in F[\mathcal{X}]$ and we fix $\phi\in C_c^\infty(\real)$ and $l(\cdot, \cdot)\in C^\infty_c (\real\times\real\setminus{\{x= y\}})$ such that $\|(\phi,l)\|_{1/(1-t)}=1$. Note that $\|\cdot\|_{\frac1{1-t}}$ and $\|\cdot\|_{\frac1{t}}$ are norms in duality as $1/\frac1{t}+1/\frac1{1-t}=1$. Define 
\begin{equation*}
\begin{aligned}
    h(z)(u,x,y) &= \begin{pmatrix}
         h_1(z)(u)\\           
         h_2(z)(x,y)
    \end{pmatrix}=
        \begin{pmatrix}
          |\phi(u)|^{(1-z)/(1-t)}\exp(i\arg\phi(u)) \\           
          |l(x,y)|^{(1-z)/(1-t)}\exp(i\arg l(x,y))
    \end{pmatrix}
     \quad \text{and} \\
    H(z) &= \int_\real\psi_1(z)(x) |\phi(x)|^{(1-z)/(1-t)}\exp(i\arg\phi(x))\,w(x)dx \\
    &\quad +\int_\real\int_\real\psi_2(z)(x,y) |l(x,y)|^{(1-z)/(1-t)}\exp(i\arg l(x,y)) \,\frac{\,dx\,dy}{|x-y|}.
\end{aligned}
\end{equation*}
%\begin{align*}
%    h(z)(u,x,y) &= \begin{pmatrix}
%         h_1(z)(u)\\           
%          h_2(z)(x,y)
%    \end{pmatrix}=
%        \begin{pmatrix}
%          |\phi(u)|^{(1-z)/(1-t)}\exp(i\arg\phi(u)) \\           
%          |l(x,y)|^{(1-z)/(1-t)}\exp(i\arg l(x,y))
%    \end{pmatrix}\\ \textrm{ and}
%  \end{align*}
%\begin{align*}
%H(z)&=\int_0^T\psi_1(z)(x) |\phi(x)|^{(1-z)/(1-t)}\exp(i\arg\phi(x))\,dx\\&\quad+\int_0^T\int_0^T\psi_2(z)(x,y) |l(x,y)|^{(1-z)/(1-t)}\exp(i\arg l(x,y)) \frac{\,dx\,dy}{|x-y|} 
%\end{align*}
Then we get \begin{align*}
H(t)&=\int_\real\psi_1(t)(x) \phi(x)\, w(x)dx+\int_\real\int_\real\psi_2(t)(x,y) l(x,y)\, \frac{\,dx\,dy}{|x-y|}.
\end{align*}
Since $H(z)$ is analytic on $S$, we can apply the three lines theorem to get 
\begin{align*}
&|H(t)|
\\&\leq\sup_{v\in\real}\left\{|H(iv) |, \quad  |H(1+iv)|\right\}\\
&\leq\sup_{v\in\real}\bigg\{\|\psi_1(iv)h_1(iv)\|_{L^1(w(x)dx)}+\|h_2(iv)\psi_2(iv)\|_{L^1(dm)},\\
&\qquad \|\psi_1(1+iv) h_1(1+iv)\|_{L^1(w(x)dx)}+\|h_2(1+iv)\psi_2(1+iv)\|_{ L^1(dm)}\bigg\}\\
&\leq \sup_{v\in\real}\bigg\{\|\psi_1(iv)\|_{L^\infty(w(x)dx)}\|h_1(iv)\|_{L^1(w(x)dx)}+\|\psi_2(iv)\|_{L^\infty(dm)}\|h_2(iv)\|_{L^1(dm)}, \\
&\qquad \|\psi_1(1+iv)\|_{L^1(w(x)dx)}\|h_1(1+iv)\|_{L^\infty(w(x)dx)}+\|\psi_2(1+iv)\|_{L^1(dm)}\|h_2(1+iv)\|_{L^\infty(dm)}\bigg\}\\
&\leq \sup_{v\in\real}\bigg\{\left(\|\psi_1(iv)\|_{L^\infty(w(x)dx)}+\|\psi_2(1+iv)\|_{L^\infty(dm)}\right) \left(\|h_1(iv)\|_{L^1(w(x)dx)}+\|h_2(iv)\|_{L^1(dm)}\right),\\
&\quad \left(\|\psi_1(1+iv)\|_{L^1(w(x)dx)}+\|\psi_2(1+iv)\|_{L^1(dm)}\right) \left(\|h_1(1+iv)\|_{L^\infty(w(x)dx)}+\|h_2(1+iv)\|_{L^\infty(dm)}\right)\bigg\}\\
&=\sup_{v\in\real}\bigg\{\|\psi_1(iv)\|_{L^\infty(w(x)dx)}+\|\psi_2(iv)\|_{L^\infty(dm)}, \|\psi_1(1+iv)\|_{L^1(w(x)dx)}+\|\psi_2(1+iv)\|_{L^1(dm)}\bigg\} 
\\&\qquad \times \sup_{v\in\real}\bigg\{\|h_1(iv)\|_{L^1(w(x)dx)}+\|h_2(iv)\|_{L^1(dm)},  \|h_1(1+iv)\|_{L^\infty(w(x)dx)}+\|h_2(1+iv)\|_{L^\infty(dm)}\bigg\}\\
&=\|\psi\|_{F[\mathcal{X}]}\|(\phi, l)\|_{\frac1{1-t}}\\
&=\|\phi\|_{\mathcal{X}_t}\|(\phi, l)\|_{\frac1{1-t}}
\end{align*}
since $\sup(ab, AB)\leq \sup(a,A)\sup(b, B)$ we can get the last two inequalities. Since $H(t)$ is the duality relation for the pair $(L^{1/t}, L^{1/{1-t}})$, we finally see 
$$\|\phi\|_{W^{1/t,s}(\real, w)}=\sup_{\|(\phi,l)\|_{1/{1-t}}=1}|H(t)|\leq \|\phi\|_{\mathcal{X}_t},$$
for all $\phi\in C_c^\infty(\real)$ finishing  this part of the proof. 
\\The identification of $\mathcal{Y}_t=W^{1/t,s-\alpha}(\real, w)$ goes along the same lines.
\end{proof}
In particular, we have the following classical Sobolev space results when we take $w=1$.
\begin{corollary} Let $ \phi\in W^{p,s}(\real)$,  $0<\alpha < s<1\leq p$, then  we have  \linebreak $\mad{}{+}{\alpha}\phi\in W^{p,s-\alpha}(\real)$.
\end{corollary}
%\begin{section}{Application}
Let $\Omega$ be a proper open subset of $\real^d$. $C_c^\infty(\Omega)$ is not dense in $W^{p,k}(\Omega)$. Instead, its closure is the space of functions $W^{p,k}_0(\Omega)$ that vanish on the boundary $\partial \Omega$. For any open set, the space $C^\infty(\Omega)\cap W^{p,k}(\Omega)$ is dense in $W^{p,k}(\Omega)$  cf. Meyers and Serrin \cite{1964_Meyers}. 
\noindent  Denote by  $W^{p,\beta}_0(0, T):=\overline{C^\infty_c(0,T)}^{\|\cdot\|_{W^{p,\beta}}} =\{\phi \in W^{p,\beta}(0, T): \phi(0\plus)=\phi(T\minus)=0\}$.
\begin{remark} If $\phi\in W^{\infty,\beta}(\Omega)$, there exists a representative $\bar\phi$ that is H\"older continuous, i.e. $\bar\phi\in C^\beta(\Omega)$, see \cite[P.13]{2023_Leoni}.
\end{remark}
We recall that the denseness of smooth compactly supported functions in fractional Sobolev space $W^{p,\beta}$
is ensured only in some cases cf. \cite[P.23]{2020_Carbotti} and Triebel \cite[P.209]{1985_Triebel}.
\begin{theorem}\label{sob-tes} Let  $\beta\in(0,1)$ and $1\leq p < \infty, \, \beta p< 1.$ Then the following result $\overline{C^\infty_c(0,T)}^{\|\cdot\|_{W^{p,\beta}}}=W^{p,\beta}(0, T)$ holds; that is , $C^\infty_c(0, T)$ is dense in $W^{p,\beta}(0,T)=W_0 ^{p,\beta}(0, T).$
\end{theorem}
As we know the results from Lemma \ref{mar-rl-frac}, Riemann-Liouville fractional derivative can be viewed as the Marchaud fractional derivative when we use the killing type extension. The next definition and theorem are taken from Lenoi \cite[P.47]{2023_Leoni} and \cite[P.32]{2013_Schmeisser}.

\begin{definition}\label{sob-ext-4-1}
Let $(0,\infty)\subseteq\real$ be an open interval, $1\leq p<\infty$, and $0<s<1$. We define the space $W_0^{p,s}(0,\infty)$ as the space of all functions $u\in W^{p,s}(0,\infty)$ such that the function $ext_\real^0 u:\real\to \real$, given by 
\begin{align*}
    ext_\real^0 u=\left\{\begin{aligned}
        &u(x), \quad x\in (0,\infty),\\
        &0, \quad \quad ~x\in \real \setminus (0,\infty),
    \end{aligned}
    \right.
\end{align*}
belongs to $W^{p,s}(\real)$.
\end{definition}

%\begin{remark}Take $\Omega=(0,T)$, then $ext_\real^0$: $W_0^{p,s}(0,T)\to W^{p,s}(\real)$ is an extension operator and $res_{(0, T)}: W^{p,s}(\real)\to W^{p,s}(0,T)$ is a restriction operator.
%\end{remark}

% \end{theorem}
% \begin{proof}
% This proof is trivial. Using the above Definition \ref{sob-ext-4-1} and Definition \ref{mar-ext}, we have the above results.
% \end{proof}
\begin{corollary}\label{rl-3-8}  Let $\phi\in W_0^{p,s}(0, T), \quad 0<\alpha <s<1$. Assume that $ext_\real^0$ , $res_{(0, T)}$ are the extension operator and restriction operator as in Definition \ref{sob-ext-4-1}. Then we have $$\rld{0}{\cdot}{\alpha}\phi=res_{(0, T)} \circ \mad{}{+}{\alpha} \circ (ext_\real^0 \phi)\in W^{p,s-\alpha}(0, T) \, i.e. \rld{0}{\cdot}{\alpha}: W_0^{p,s}(0, T) \to W^{p,s-\alpha}(0, T) .$$
\end{corollary}
%\begin{definition}\label{sob-ext-4-2}
%Let $\Omega\subseteq\real$ be an open interval, $1\leq p<\infty$, and $0<s<1$. We define the space $W_0^{p,s}(\Omega)$ as the space of all functions $u\in W^{p,s}(\Omega)$ such that the function $ext_\real^0 u:\real\to \real$, given by 
%\begin{align*}
%    ext_\real^0 u=\left\{\begin{aligned}
%        &u(x), \quad x\in \Omega,\\
%        &0, \quad \quad ~x\in \real \setminus \Omega,
%    \end{aligned}
%    \right.
%\end{align*}
%belongs to $W^{p,s}(\real)$.
%\end{definition}
%\end{theorem}
As we know the results from Lemma \ref{mar-rl-frac}, the Caputo fractional derivative can be viewed as the Marchaud fractional derivative when we use the sticky type extension. We have the following results  that the even extension operator maps fractional Sobolev space $W^{p,s}((0,\infty))$ to $W^{p,s}(\real)$.
\begin{theorem}\label{sob-ext-4-2} Let $(0,\infty)\subseteq\real$ be an open interval and %$w\in W(\alpha)$, 
$1\leq p<\infty,\, ps\geq1$,  and define 
$ext_\real^e u:\real\to \real$, given by 
\begin{align*}
    ext_\real^e u=\left\{\begin{aligned}
        &u(x), \quad x\in (0,\infty),\\
        &u(-x), \quad \quad ~x\in \real \setminus (0,\infty),
    \end{aligned}
    \right.
\end{align*}
 we have $ext_\real^e: W^{p,s}((0,\infty))\to W^{p,s}(\real)$, which is a continuous linear map. %and $res_{(0, T)}^w: W^{p,s}(\real, w)\to W^{p,s}((0,T),w)$ is a restriction operator.
\end{theorem}
\begin{proof}
In the view of Theorem 2.18 in \cite[P. 62]{2023_Leoni}, it is trivial to show $ext_\real^e u\in W^{p,s}(\real)$.
When $ps=1$,  $$\int_0^\infty\left|u(x)-u[-(-x)]\right|^p \,\frac{dx}{x}=0<\infty.$$
When $ps>1$, $\bar{u}(0)=\bar{u}(-0)$, where $\bar{u}$ is the H\"{o}lder continuous
representative of $u$.
\end{proof}
\begin{corollary}\label{rl-3-9}  Let $\phi\in W^{p,s}(0, \infty), \quad 0<\alpha <s<1< p$ and $ps\geq1$. Assume that $ext_\real^e$ is the extension operator as in Theorem \ref{sob-ext-4-2}. %$res_{(0, T)}^w$ are the extension operator and restriction operator as in Theorem \ref{sob-ext-4-2}.
 Then we have $$\ed{0}{x}{\alpha}\phi= \mad{}{+}{\alpha} \circ (ext_\real^e \phi)\in W^{p,s-\alpha}(\real), i.e. \ed{0}{x}{\alpha}:  W^{p,s}((0,\infty))\to W^{p,s-\alpha}(\real).$$ %\, i.e. $\cd{0}{\cdot}{\alpha}: W^{p,s}((0, T), w) \to W^{p,s-\alpha}((0, T),w) $. 
\end{corollary}
\subsection*{Declaration of interests}
{The authors report no conflict of interest.  No datasets were used.}

\bibliographystyle{plain}
\bibliography{Thesis.bib}
\end{document}